\documentclass[12pt,a4paper]{amsart}
\usepackage{latexsym,amsfonts,amsthm,amsmath,mathrsfs,amssymb}
\usepackage[british]{babel}
\usepackage[dvips]{color}
\usepackage[utf8]{inputenc}
\usepackage[T1]{fontenc}
\usepackage[dvips,final]{graphics}
\usepackage{xcolor}
\usepackage{mathrsfs}
\usepackage{mathtools} 
\usepackage{breqn}
\usepackage{epstopdf}
\usepackage{verbatim}
\usepackage{amscd}
\usepackage{hyperref}
\usepackage[english,capitalize]{cleveref}
\usepackage{aliascnt}

\usepackage{todonotes}

\makeatletter
\def\newaliasedtheorem#1[#2]#3{
	\newaliascnt{#1@alt}{#2}
	\newtheorem{#1}[#1@alt]{#3}
	\expandafter\newcommand\csname #1@altname\endcsname{#3}
}
\makeatother

\theoremstyle{plain}

\newtheorem{theorem}{Theorem}[section]
\newaliasedtheorem{prop}[theorem]{Proposition}
\newaliasedtheorem{lem}[theorem]{Lemma}
\newaliasedtheorem{coroll}[theorem]{Corollary}

\theoremstyle{definition}
\newaliasedtheorem{defi}[theorem]{Definition}
\newaliasedtheorem{quest}[theorem]{Question}
\newaliasedtheorem{fact}[theorem]{Fact}

\theoremstyle{remark}
\newaliasedtheorem{rem}[theorem]{Remark}
\newaliasedtheorem{exa}[theorem]{Example}

\newcommand{\R}{\mathbb{R}}
\newcommand{\N}{\mathbb{N}}

\newcommand{\W}{\mathbb{W}}

\newcommand{\C}{\mathbb{C}}

\newcommand{\G}{\mathbb{G}}

\newcommand{\mm}{{\mathfrak D}}

\let\altphi\phi
\let\phi\varphi
\let\varphi\altphi
\let\altphi\undefined

\newcommand{\average}{{\mathchoice {\kern1ex\vcenter{\hrule height.4pt
width 6pt
depth0pt} \kern-9.7pt} {\kern1ex\vcenter{\hrule height.4pt width 4.3pt
depth0pt}
\kern-7pt} {} {} }}

\address{\textsc{Daniela Di Donato}: 
Dipartimento di Ingegneria Industriale e Scienze Matematiche, Via Brecce Bianche, 12 60131 Ancona, Universit\'a Politecnica delle Marche.}
\email{d.didonato@staff.univpm.it}

\title{Non-symmetric intrinsic Hopf-Lax semigroup vs. intrinsic Lagrangian}

\date{\today}

\author{ Daniela Di Donato }

\begin{document}

\begin{abstract}

		In this paper, we analyze the 'symmetrized' of the intrinsic Hopf-Lax semigroup introduced by the author in the context of the intrinsically Lipschitz sections in the setting of metric spaces. Indeed, in the usual case, we have that $d(x,y) =d(y,x)$ for any point $x$ and $y$ belong to the metric space $X;$ on the other hand, in our intrinsic context, we have that $d(f(x),\pi^{-1} (y)) \ne d(f(y),\pi^{-1} (x)),$ for every $x,y \in X.$ Therefore, it is not trivial that we get the same result obtained for the "classical" intrinsic Hopf-Lax semigroup, i.e., the 'symmetrized'  Hopf-Lax semigroup is a subsolution of Hamilton-Jacobi type equation. Here, an important observation is that $f$ is just a continuous section of a quotient map $\pi$ and it can not intrinsic Lipschitz.  
		
		However, following Evans, the main result of this note is  to show that the "new" intrinsic Hopf-Lax semigroup satisfies a suitable variational problem where the functional contained an intrinsic Lagrangian. Hence,  we also define and prove some basic properties of the  intrinsic Fenchel-Legendre transform of this intrinsic Lagrangian that depends on a continuous section of $\pi$.
		
		\end{abstract}

\maketitle 
\tableofcontents

\section{Introduction} 	
Find out a good notion of rectifiability in subRiemannian Carnot groups \cite{ABB, BLU, CDPT} is unanswered questions which have continued to perplex many mathematicians for decades. Starting from a negative result by Ambrosio and Kirchheim \cite{AmbrosioKirchheimRect} (see also \cite{MAGNANI1}), Franchi, Serapioni and Serra Cassano introduced the notion of intrinsically Lipschitz maps \cite{FSSC, FSSC03, MR2032504}  (see also  \cite{SC16, FS16}) in order to give positive results to  this question. This is one possible positive solution; the reader can also see \cite{AM2022.1, AM2022.2, DB22, DS91a, DS91b, DD19, MR2247905, MR2048183, NY18}. 

The simple idea of intrinsic graph in the FSSC sense is the following one: let $\mathbb{V}$ and $\W$ be complementary homogeneous subgroups of a Carnot group $\G$, i.e., $\W \cap \mathbb{V}= \{ 0 \}$ and $\G=\W\cdot \mathbb{V}$ (where $\cdot$ is the group operation given by the Campbell-Hausdorff formula), then the intrinsic left graph of $\phi :\W\to \mathbb{V}$ is the set $$ \mbox{graph} {(\phi)}:=\{ x\cdot \phi (x) \, |\, x\in \W \}.$$ 
 A function $\phi $ is said to be intrinsic Lipschitz if it is possible to put, at each point $p\in \mbox{graph} {(\phi)}$, an intrinsic cone with vertex $p$, axis $\mathbb{V}$ and fixed opening, intersecting $\mbox{graph} {(\phi)}$ only at $p$.

Moreover, in \cite{FSSC, MR2032504}, the authors introduce the notion of  intrinsic $C^1$ surface adapting to groups De Giorgi's classical technique 
 valid in Euclidean spaces to show that the boundary of a finite perimeter set can be seen as a countable union of $C^1$ regular surfaces. A set $S$ is a $d$-codimensional intrinsic $C^1$ surface   if there exists a continuous function $f:\G \to \R^d$ such that,  locally, $$S=\{ p\in \G : f(p)=0\},$$ and the horizontal jacobian of $f$ has maximum rank, locally (see also \cite{DD19}).
  
These two approaches are natural counterparts of the notions of rectifiability in Euclidean spaces, where their  equivalence follows from well-known theorems as, for instance, Rademacher Theorem. Hence, it is surprising that the connection between these two notions of rectifiability is poorly  understood already in  Carnot groups of step 2. \\

Starting from the seminal papers by Franchi, Serapioni and Serra Cassano, in metric spaces, Le Donne and the author \cite{DDLD21} introduced and studied the so-called intrinsically Lipschitz sections (see Section \ref{Intrinsically   Lipschitz  sections.21}) which generalize the intrinsically Lipschitz maps in the FSSC sense. 
  The difference between the two approaches is that Franchi, Serapioni and Serra Cassano study the properties of intrinsically Lipschitz maps; while we study the ''sections'' or rather the properties of the graphs that are intrinsic Lipschitz. Roughly speaking, in our approach, a section $\psi$ is such that $\mbox{graph}(\phi) =\psi (Y) \subset X$ where $X$ is a metric space and $Y$ is a topological space. In \cite{DDLD21}, we prove the Ascoli-Arzel\'a theorem, the Ahlfors-David regularity, the link between intrinsic Lipschitz sections and level sets of biLipschitz maps on fibers, as well as the Extension theorem. In \cite{D22.1}, we also prove other properties, following Cheeger's idea \cite{C99} (see also \cite{KM16, K04}), like: convexity and being vector space over $\R$ or $\C$ for a suitable class of these sections; we give an equivalence relation for these sections.
%
%
%
%

In our context a section $\phi :Y \to X$ of  $\pi:X \to Y$ is such that $\pi$ produces a foliation for $X,$ i.e., $X= \coprod \pi ^{-1} (y)$ and the Lipschitz property of $\phi$ consists to ask that the distance between two points $\phi (y_1), \phi (y_2)$ is not comparable with the distance between $y_1$ and $y_2$ but between $\phi (y_1)$ and the fiber of $y_2$.  Following this idea and recall that the Hopf-Lax formula
\begin{equation*}
F(y,t)= \inf_{z\in M} \left\{ f(z) + \frac{d^2(y,z)}{2t} \right\},
\end{equation*}
is the viscosity solution to the Hamilton-Jacobi equation on a compact Riemannian manifold $M,$ i.e.,
  \begin{equation*}
\frac{\partial F}{\partial t} + \frac{|\nabla F|^2}{2}=0.
\end{equation*}
(where $f:M \to \R$ is the initial continuous condition, $d$ is the geodesic distance on $M$ and $t\in \R^+),$ it is natural to define the intrinsic Hopf-Lax semigroup as follows. Let $X=\R^\kappa$, $Y\subset \R^\kappa$ be bounded and $\pi :X \to Y$ a quotient map. 
The intrinsic Hopf-Lax semigroup is the family of operators $(iQ_\cdot)_{t>0}$ defined as 
\begin{equation}\label{equation231.old}
f \,\, \mapsto \,\, iQ_tf(y) := \inf_{z\in Y} \left\{ \max_{j=1,\dots , \kappa} f_j(z) +\frac 1 {2t} d^2(f(z), \pi^{-1} (y)) \right\}.
\end{equation}
for any continuous section $f:Y \to X$ of $\pi.$ We also consider the case $\kappa >1$ because being intrinsically Lipschitz is equivalent to Lipschitz in the classical sense  when we consider the basic case $X=Y=\R.$ Moreover, I underline that, in most cases, $f$ is just a continuous section of $\pi$ and hence we do not ask for its intrinsic Lipschitz condition.

This is the topic of \cite{D22Hopf} where, following  \cite{AGS08.2, LV07}, we prove the  link between the intrinsic Hopf-Lax semigroup and  the intrinsic slope and we also show that the intrinsic Hopf-Lax semigroup is a subsolution of Hamilton-Jacobi type equality. Yet, in the usual case, the first result is one step in order to analyze  when a Dirichlet form is regular (i.e., when it coincides with the Cheeger energy) and thus a natural question may be whether the results known in \cite{ACDM15, AES16, AGS08.1,  AGS08.2, AGS08.3, AGS08.4, BGL01, KSZ14, KZ12, FOT10, Savar1} can be fitted into our intrinsic context.

However, in this paper, we analyze a sort of 'symmetrized' of \eqref{equation231.old}. Indeed, in the usual case, we have that the symmetric property of $d$ is true, i.e., $d(x,y) =d(y,x)$ for any point $x,y \in X;$ on the other hand, in our intrinsic context, it is easy to see that $$d(f(x),\pi^{-1} (y)) \ne d(f(y),\pi^{-1} (x)),$$ for every $x,y \in X.$ In particular, it holds
\begin{equation}\label{equationDIFFERENZA.22}
\begin{aligned}
d(f(y), \pi^{-1} (x)) -d(f(z), \pi^{-1} (x))  & \leq d(f(y),f(z)), \quad \forall x,y,z \in Y\\
d(f(x), \pi^{-1} (y)) -d(f(x), \pi^{-1} (z)) & \nleq  d(f(y),f(z)), \quad \mbox{for some } x,y,z \in Y.\\
\end{aligned}
\end{equation}

Indeed, for any fixed $x,y,z \in Y,$ if we choose $a \in \pi^{-1} (x)$ such that
\begin{equation*}
d(f(z), \pi^{-1} (x)) =d(f(z), a),
\end{equation*}
we get that
\begin{equation}\label{equation23.servira} \begin{aligned}
d(f(y), \pi^{-1} (x)) -d(f(z), \pi^{-1} (x)) & \leq  d(f(y),a)  -d(f(z), \pi^{-1} (x)) \\ & \leq   d(f(y),f(z)) +d(f(z), a ) -d(f(z), \pi^{-1} (x)) \\
& = d(f(y),f(z)),\\
\end{aligned}
\end{equation}  i.e., the first inequality of \eqref{equationDIFFERENZA.22} holds. 
On the other hand, for the second inequality in \eqref{equationDIFFERENZA.22}, we give the following example. let $X\subset \R^2$ the set given by the two lines with vertex $(0,8), (8,8)$ and $(0,3),(8,7)$ and the subset $Y$ of $\R^2$ defined as the line with vertex $(0,0)$ and $(8,0).$ If we consider  a continuous section $f:Y \to X$ of the projection $\pi :X\to Y$ with $f(x)=f((1,0))=(1,4), f(y)=f((7,0))=(8,7)$ and $f(z)=f((6,0))=(8,6),$ it is easy to see that
\begin{equation*} \begin{aligned}
d(f(x), \pi^{-1} (y)) -d(f(x), \pi^{-1} (z)) & = \sqrt{\frac{5}{4}}, \\
 d(f(y),f(z)) &= 1,
\end{aligned}
\end{equation*} 
and so \begin{equation*} \begin{aligned}
d(f(x), \pi^{-1} (y)) -d(f(x), \pi^{-1} (z)) & \nleq  d(f(y),f(z)).\\
\end{aligned}
\end{equation*}

Then, it is not trivial to define 
\begin{equation}\label{equationHOPFLAXSYMMETRIZZATO}
u(y,t) = \inf_{z\in Y} \left\{ \frac{d^2(f(y), \pi^{-1} (z))}{2t}+g(z) \right\}.
\end{equation}
and to prove that it is a subsolution of Hamilton-Jacobi type equation (see Corollary \ref{provaHJno}). 

However, the main object of this paper is defined in \eqref{equationu21} and \eqref{equationHOPFLAXSYMMETRIZZATO} is our model case. 
The main results of this paper are Theorem \ref{theoremLagrangian19} and Proposition \ref{proptransformFL} (3). Here, first we define a suitable map $u$ (see \eqref{equationu21}) which is different to \eqref{equation231.old} because we change $ d^2(f(z), \pi^{-1} (y))$ instead of  $ d^2(f(y), \pi^{-1} (z)),$ for $y,z \in X=\R^\kappa .$ Then, following  \cite{DB09, E10}, we prove that $u$ is solution of the following variational problem
\begin{equation*}
\inf_{w \in C(\R)} \left\{ \int _0^t L(\dot w(s)) \, ds + g(w(0))\,|\, w(t)= w(0) + d(f(y) , \pi^{-1} (w(0))) \right\},
\end{equation*}
and satisfies the boundary condition $u=g \mbox{ on } Y \times \{t=0\},$ where $g:=\max_{j=1,\dots , \kappa } f_j.$ In the usual case, $L$ is the Lagrangian of $u$ and so we call it "intrinsic Lagrangian". We ask nothing regarding asymptoticity of $L$ but we ask its convexity. However, the reader can see Section \ref{Properties of L19} about the properties of our intrinsic Lagrangian which in our model case corresponds to $ L(v)=v^2$ with $v\geq 0.$
 
Finally, following \cite{DB09, E10}, we define and prove some basic properties of the intrinsic Fenchel-Legendre transform as follows (see Section \ref{The intrinsic Fenchel-Legendre transform}).
\begin{defi}
Let $y\in Y, t>0$ and  $f: Y \to \R^\kappa$ be a continuous section of $\pi.$ The intrinsic Fenchel-Legendre transform of $L$ is a map $L^*: [0,Ils(f)] \to \R$ given by 
\begin{equation}
 L^*(\xi)= L^*_{f, y,t}(\xi):= \sup_{\substack{ z\in Y\, \\ w= \frac{d(f(y), \pi^{-1} (z))}{t} \,  \in \R^+ } } ( \xi w- L\left(w\right)).
\end{equation}
\end{defi}
As in classical way, we define the intrinsic Hamiltonian $H$ associated by $L$ as the  intrinsic Fenchel-Legendre transform of $L$ and so $$H(\xi) :=L^*(\xi), $$  for any $\xi \in [0,Ils(f)].$

 {\bf Acknowledgements.}  We would like to thank Professor Giuseppina Autuori for the references \cite{DB09, E10}.

\section{Intrinsically   Lipschitz  sections}\label{Intrinsically   Lipschitz  sections.21}
\subsection{Intrinsically   Lipschitz  sections} The notion of intrinsically Lipschitz maps was introduced by Franchi, Serapioni and Serra Cassano  \cite{FSSC, FSSC03, MR2032504}  (see also  \cite{SC16, FS16}) in the context of subRiemannian Carnot groups after a negative result given by Ambrosio and  Kirchheim \cite{AmbrosioKirchheimRect}. Their aim is to establish a good definition of rectifiability in Carnot groups.

Here we present a generalization of this concept introduced in \cite{DDLD21}. Our setting is the following. We have a metric space $X$, a topological space $Y$, and a 
quotient map $\pi:X\to Y$, meaning
continuous, open, and surjective.
The standard example for us is when $X$ is a metric Lie group $G$ (meaning that the Lie group $G$ is equipped with a left-invariant distance that induces the manifold topology), for example a subRiemannian Carnot group, 
and $Y$ if the space of left cosets $G/H$, where 
$H<G$ is a  closed subgroup and $\pi:G\to G/H$ is the projection modulo $H$, $g\mapsto gH$.
\begin{defi}\label{def_ILS} 
We say that a map $\phi :Y \to X$ is a section of $\pi $ if 
\begin{equation}\label{equation1}
\pi \circ \phi =\mbox{id}_Y.
\end{equation}
Moreover, we say that a map $\phi:Y\to X$ is an {\em intrinsically Lipschitz section of $\pi$ with constant $L$},  with $L\in[1,\infty)$, if in addition
\begin{equation}\label{equationFINITA}
d(\phi (y_1), \phi (y_2)) \leq L d(\phi (y_1), \pi ^{-1} (y_2)), \quad \mbox{for all } y_1, y_2 \in Y.
\end{equation}
Here $d$ denotes the distance on $X$, and, as usual, for a subset $A\subset X$ and a point $x\in X$, we have
$d(x,A):=\inf\{d(x,a):a\in A\}$.
\end{defi}
A first observation is that this class is contained in the class of continuous maps (see \cite[Section 2.4]{DDLD21}) but  cannot be uniformly continuous (see Example 1.2 in \cite{D22.24june}). 
 In the case  $ \pi$ is a Lipschitz quotient or submetry \cite{MR1736929, Berestovski},  being intrinsically Lipschitz  is equivalent to biLipschitz embedding, see Proposition 2.4 in \cite{DDLD21}. Moreover, since $\phi$ is injective by \eqref{equation1}, the class of Lipschitz sections not include the constant maps.

   \begin{rem}
   If $Y$ is bounded, we get that \begin{equation}\label{costuniversale}
K:= \sup _{y_1,y_2 \in Y}  d(\phi (y_1), \pi ^{-1} (y_2)) <+ \infty .
\end{equation}
This follows because, on the contrary, if $K=+ \infty,$ then we get the contradiction $+ \infty= d(\phi (y_1), \pi ^{-1} (y_2))  \leq d(\phi (y_1), \phi (y_2)).$
  \end{rem}

\subsection{Intrinsically   Lipschitz  constants}
 We recall the definition of the intrinsically Lipschitz constants as in \cite{D22.31may, D22.29}, where we have adapted the theory of  \cite{C99, DM} in our intrinsic case.

\begin{defi}\label{def_ILS.1} Let $\phi:Y\to X$ be a section of $\pi$. Then we define 
\begin{equation*}
ILS (\phi):= \sup _{\substack{y_1, y_2 \in Y \\ y_1\ne y_2}} \frac{d(\phi (y_1), \phi (y_2))}{  d(\phi (y_1), \pi ^{-1} (y_2)) } \in [0, + \infty ]
\end{equation*}
and
\begin{equation*}
\begin{aligned}
ILS (Y,X,\pi ) &:= \{ \phi :Y \to X \,:\, \phi \mbox{ is an intrinsically Lipschitz section of $\pi$ and }  ILS(\phi) < +\infty \},\\
ILS_{b} (Y,X,\pi) & := \{ \phi \in  ILS (Y,X,\pi) \,:\, \mbox{spt}(\phi) \mbox{ is bounded} \}.\\
\end{aligned}
\end{equation*}
\textbf{For simplicity, we will write $ILS (Y,\R ^\kappa)$ instead of  $ILS (Y,\R ^\kappa ,\pi ).$}
\end{defi}

\begin{defi}\label{def_ILS.2} Let $\phi:Y\to X$ be a  section of $\pi$. Then we define the local intrinsically Lipschitz constant (also called intrinsic slope) of $\phi$ the map $Ils (\phi):Y \to [0,+\infty )$ defined as 
\begin{equation*}
Ils (\phi) (z):= \limsup _{y\to z} \frac{d(\phi (y), \phi (z))}{  d(\phi (y), \pi ^{-1} (z)) },
\end{equation*}
if $z \in Y$ is an accumulation point; and $Ils (\phi) (z):=0$ otherwise.
\end{defi}

\begin{defi}\label{def_ILS.3} Let $\phi:Y\to X$ be a section of $\pi$. Then we define the asymptotic  intrinsically Lipschitz constant of $\phi$ the map $Ils_a (\phi):Y \to [0,+\infty )$ given by
\begin{equation*}
Ils_a (\phi) (z):= \limsup _{y_1,y_2\to z}\frac{d(\phi (y_1),\phi (y_2))}{  d(\phi (y_1), \pi ^{-1} (y_2)) }
\end{equation*}
if $z \in Y$ is an accumulation point and $Ils (\phi) (z):=0$ otherwise.
\end{defi}
 
 \begin{rem}\label{defrem} Notice that by $\phi (y_2) \in \pi ^{-1} (y_2),$ it is trivial that $ d(\phi (y_1), \pi ^{-1} (y_2)) \leq d(\phi (y_1), \phi (y_2))$ and so $Ils(\phi) \geq 1.$ Moreover, it holds 
\begin{equation*}
 Ils (\phi ) \leq Ils_a (\phi ) \leq ILS(\phi).
\end{equation*}
\end{rem}

\section{The intrinsic Lagrangian $L$}

\subsection{Properties of L}\label{Properties of L19}
We underline that our standard $L$ has the following form:
\begin{equation*}
tL\left(\frac{d(f(x), \pi^{-1} (z))}{t} \right) = \frac{d^2(f(x), \pi^{-1} (z))}{2t}, 
\end{equation*} for any $x,z \in Y.$ Hence, we want that $L$ satisfies the following three properties:
\begin{itemize}
\item As in the classical case, we ask that $L$ is convex map.
\item It holds
\begin{equation}\label{equationPorpL}
tL\left(\frac{d(f(y), \pi^{-1} (z))}{t} \right) - tL \left(\frac{d(f(x), \pi^{-1} (z))}{t} \right) \leq  2K\sqrt L\left(\frac{d(f(y),f(x))}{t}\right),
\end{equation} for any $x,y ,z \in Y$ and $t>0.$  
This is because the following observation: for any $x,y,z \in Y,$ by \eqref{equation23.servira}, we know that  $d(f(y), \pi^{-1} (x)) -d(f(z), \pi^{-1} (x))  \leq d(f(y),f(z)), \quad \forall x,y,z \in Y$  and so if $Y$ is bounded, then we have that 
\begin{equation*}
d^2(f(y), \pi^{-1} (z)) -d^2(f(x), \pi^{-1} (z)) \leq  2Kd(f(y),f(x)),
\end{equation*}
 for any $x,y, z \in Y,$ where $K>0$ is given by \eqref{costuniversale}.
 \item As in the usual case, by the simply fact that if $0<s<t,$ then $1/s>1/t$ we ask
\begin{equation}\label{equationPorpL.2}
tL\left(\frac{d(f(x), \pi^{-1} (z))}{t} \right)  \leq  s L \left(\frac{d(f(x), \pi^{-1} (z))}{s} \right),
\end{equation} for any $x,z \in Y.$
\end{itemize}

\subsection{Properties of the solution $u$}
The purpose of this section is to follows the theory in  \cite{DB09, E10} where the authors  investigate the sense in which function $\ell$ defined as 
\begin{equation*}
\ell(y,t) := \inf \left\{ \int _0^t L(\dot w(s)) \, ds + g(w(0))\,|\, w(t)=y \right\},
\end{equation*}

 (see \cite[Section 3.3]{E10}) actually solves the initial-valued problem for the Hamilton-Jacobi PDE:
\begin{equation}\label{CondChowVerificata8}
\left\{
\begin{array}{l}
\ell_t +H(D\ell)=0 \quad \mbox{in } \R^\kappa \times (0,\infty ), \\ 
\ell=g \quad \mbox{on } \R^\kappa \times \{t=0\}, \\ 
\end{array}
\right.
\end{equation}
where $H$ is a smooth map such that
\begin{enumerate}
\item $H$ is convex;
\item $\lim_{|p|\to \infty } \frac{ H(p)}{|p|} =+\infty;$
\end{enumerate}
and $g:\R^\kappa \to \R$ is a Lipschitz map in the usual sense.

The main aims of this paper is to adapt this theory in our intrinsic context. More precisely, we prove that a suitable map $u$ defined as \eqref{equationu21} is a subsolution of Hamilton-Jacobi type equation and that this map is equal to $g$ when $t=0$ (see Corollary \ref{provaHJno} and Proposition \ref{proptransformFL} (3)). Moreover, in Theorem \ref{theoremLagrangian19} we show that this map $u$ is also a solution of a suitable variational problem (see \eqref{equationPB17}). 
\\

Our setting is as follows. Let $\pi :\R^\kappa \to Y$ be a quotient map with $Y$ a bounded subset  of $\R^\kappa.$  We consider a continuous section $f:Y \to \R^\kappa$ such that
\begin{equation*}
g(y):= \max _{j=1,\dots , \kappa } f_j(y),
\end{equation*}
where if $f$  is an intrinsically Lipschitz section of $\pi,$ then $g$ is so too with the same Lipschitz constant (see \cite[Proposition 4.1]{D22.29}). However, we present some result where $f$ is just a continuous section of $\pi.$

Given $y\in Y$ and $t>0,$ we propose to minimize among curves $w:\R^+ \cup \{ 0\} \to \R$ satisfying 
\begin{equation}
w(t)= w(0) +d(f(y) , \pi^{-1} (w(0))),
\end{equation}
 the expression 
 \begin{equation}
\int _0^t L(\dot w (s)) \, ds + g(w(0)),
\end{equation}
which is the action augmented with the value of the initial data. We accordingly now define 
\begin{equation}\label{equationPB17}
v(y,t) := \inf \left\{ \int _0^t L(\dot w(s)) \, ds + g(w(0))\,|\, w(t)= w(0) + d(f(y) , \pi^{-1} (w(0))) \right\},
\end{equation}
the infimum taken over all $C^1$ maps $w(.).$

\begin{theorem}\label{theoremLagrangian19}
If $y\in Y$ and $t>0,$ then the solution $u=u(y,t)$ of the minimization problem \eqref{equationPB17} is
\begin{equation}\label{equationu21}
u(y,t) = \inf_{z\in Y} \left\{ tL\left(\frac{d(f(y), \pi^{-1} (z))}{t} \right)+g(z) \right\}.
\end{equation}
\end{theorem}

\begin{proof}
Fix $z\in Y$ and define $\tilde w (s):= z + \frac s t d(f(y) , \pi^{-1} (z))$ for $s\in [0,t].$ By definition of $v$ we deduce that
\begin{equation*}
v(y,t) \leq  \int _0^t L(\dot {\tilde w}(s)) \, ds + g(z) = t L\left(\frac{d(f(y), \pi^{-1} (z))}{t} \right)+g(z),
\end{equation*}
and so $v(y,t) \leq u(y,t).$

On the other hand, if $w(.)$ is any $C^1$ map satisfying $w(t)= w(0) +d(f(y) , \pi^{-1} (w(0)))$, using the convexity of $L,$ we have
\begin{equation*}
L\left( \frac 1 t \int _0^t  \dot { w}(s) \, ds \right) \leq  \frac 1 t \int _0^t L(\dot { w}(s)) \, ds,
\end{equation*}
by Jensen's inequality. Thus if we write $z=w(0)$ we find 
\begin{equation*}
t L\left(\frac{d(f(y), \pi^{-1} (z)) } {t} \right) +g(z) \leq  \frac 1 t \int _0^t L(\dot { w}(s)) \, ds + g(z).
\end{equation*}
Consequently $v(y,t) \geq u(y,t)$ and the proof is complete.
\end{proof}

When $f$ is a bounded map, we get the following observation.

 \begin{prop}\label{theoremHopfLax}
Let $f \in C_b(Y,\R ^\kappa)$ be a section of  $\pi :\R^\kappa \to Y$. Then we have the following basic properties for $u(y,t):$
\begin{description}
\item[ i]  $\inf _{z\in Y} \min_{j=1,\dots , \kappa} f_j \leq u(\cdot, t) \leq \max_{j=1,\dots , \kappa} f_j \leq \sup _{z\in Y} \max_{j=1,\dots , \kappa} f_j< + \infty.$
\item[ ii]  Any  quasi-minimizing sequence $(y_n)_n$ for $u(y,t_n)$  converges to $y$ as $t_n\to 0$ and 
\begin{equation}\label{equation24giugno}
 \pi^{-1} (y_n) \to  \pi^{-1} (y), 
\end{equation} as $n\to \infty .$
\end{description}

\end{prop}

\begin{proof} 
$i)$  It is a trivial consequence of the fact that $\frac 1 {2t} d^2(f(z), \pi^{-1} (z)) \geq 0$ and that we can use $z=y$ as a competitor in the infimum of  \eqref{equationu21}.

$ii)$ Fix $y\in Y$ and take a sequence $t_n \to 0$; consider a quasi-minimizing sequence $(y_n)_n$ for $u(y,t_n)$, in the sense that:
\begin{equation*}
u(y,t_n) +\frac 1 n \geq g(y_n) + \frac 1 {2t_n} d^2(f(y) , \pi^{-1} (y_n)), \quad \forall n\in \N .
\end{equation*} 
Indeed,
\begin{equation*}
\begin{aligned}
d^2 ( \pi^{-1} (y) , \pi^{-1} (y_n)) & \leq d^2 (f(y) , \pi^{-1} (y_n))\\
&  \leq 2 t_n \left( u(y,t_n) +\frac 1 n - g(y_n) \right) \\
& \leq 2t_n \left(2\|f \|_\infty +\frac 1 n \right) \,\, \longrightarrow _{n\to \infty } 0.
\end{aligned}
\end{equation*}
 Consequently, \eqref{equation24giugno} holds by continuity of $f$.

\end{proof}

\begin{prop}\label{proptransformFL}
The function $u$ defined as in \eqref{equationu21} satisfies the following properties:
\begin{enumerate}
\item If $Y$ is bounded, it holds
\begin{equation*}
|u(x,t) -u(y,t)| \leq  2K\sqrt L\left(\frac{d(f(y),f(x))}{t}\right), \quad \forall x,y \in Y, \, t>0, \end{equation*} where $K>0$ is given by \eqref{costuniversale}.
\item If $Y$ is bounded, it holds
\begin{equation*}
u(y,t) \leq 2K \sqrt L\left(\frac{d(f(y),f(x))}{t}\right)+u(x,s), \quad \forall x,y \in Y, \, 0<s<t. \end{equation*}
\item If $f$ is an intrinsically Lipschitz map, then
\begin{equation*}
u=g, \quad \mbox{ on } Y\times \{t=0\}.
\end{equation*}
\item it holds for any $s,t \in \R^+$ such that $s<t$
\begin{equation*}
u(y,t)-u(y,s) \leq 0,
\end{equation*}
for every $y\in Y.$
\end{enumerate}

\end{prop}

\begin{proof}
$(1).$ Let $x,y \in Y$ and $t>0.$ We choose $z\in Y$ such that $u(y,t) = tL\left(\frac{d(f(y), \pi^{-1} (z))}{t} \right)+g(z)$ and so by \eqref{equationPorpL} we get 
\begin{equation*}
\begin{aligned}
u(x,t) - u(y,t ) & = \inf_{h\in Y} \left\{ tL\left(\frac{d(f(x), \pi^{-1} (h))}{t} \right)+g(h ) \right\} - tL\left(\frac{d(f(y), \pi^{-1} (z))}{t} \right)-g(z)\\
& \leq  t \left(L\left(\frac{d(f(x), \pi^{-1} (z))}{t} \right)- L\left(\frac{d(f(y), \pi^{-1} (z))}{t} \right) \right)\\
& \leq 2K\sqrt L\left(\frac{d(f(y),f(x))}{t}\right).
\end{aligned}
\end{equation*}
Hence, interchanging the roles of $x$ and $y$  we have the thesis.

$(2).$   Let $x,y \in Y$ and $0<s<t.$ We choose $z\in Y$ such that $u(x,t) = tL\left(\frac{d(f(x), \pi^{-1} (z))}{t} \right)+g(z)$ and so by \eqref{equationPorpL} we get 
\begin{equation*}
\begin{aligned}
u(y,t) - u(x,s )  
& \leq  t L\left(\frac{d(f(y), \pi^{-1} (z))}{t} \right)- s L\left(\frac{d(f(x), \pi^{-1} (z))}{s} \right)\\
& \leq 2K  \sqrt L\left(\frac{d(f(y),f(x))}{s}\right)
\end{aligned}
\end{equation*}
where in the last inequality we used \eqref{equationPorpL.2} and \eqref{equationPorpL}.

$(3).$  Let $x \in Y$ and $t>0.$ Choosing $z=x$ we obtain
\begin{equation}\label{equation29}
u(x,t) \leq tL(0) +g(x).
\end{equation}
Moreover, using the intrinsic Lipschitz property of $f$ and so of $g$ we have that
\begin{equation*}
\begin{aligned}
u(x,t) & = \inf_{z\in Y} \left\{ tL\left(\frac{d(f(x), \pi^{-1} (z))}{t} \right)+g(z) \pm g(x) \right\}\\
& \geq g(x) +  \inf_{z\in Y} \left\{- Ils(f) d(f(y), \pi^{-1} (z))+ tL\left(\frac{d(f(y), \pi^{-1} (z))}{t} \right) \right\}\\
& = g(x) -t \max_{\substack{ z\in Y\, \\ w= \frac{d(f(y), \pi^{-1} (z))}{t} }}  ( Ils(f) w- L\left(w\right)) \\
& = g(x)  -t \max_{\xi \in [0,Ils(f)]} \max_{\substack{ w= \frac{d(f(y), \pi^{-1} (z))}{t} }}  ( \xi w- L\left(w\right)) \\
& =: g(x)  -t \max_{\xi \in [0,Ils(f)]} H(\xi),\\
\end{aligned}
\end{equation*}  where in the first inequality we used the fact that if $f$  is an intrinsically Lipschitz section of $\pi,$ then $g$ is so too. 
In the next section we explain what is $H;$ more precisely, it is the Hamiltonian associated by $L$ in our intrinsic context. Finally, putting together the last inequality and \eqref{equation29}, it holds
\begin{equation*}
|u(x,t) - g(x)| \leq \max\{ |L(0)| , \max _{\xi \in [0,Ils(f)]} |H(\xi)|\} t =: C t,
\end{equation*}
which implies that $u=g$ on $Y \times \{ t=0 \},$ as desired.

$(4).$  Let $y \in Y$ and $t>0.$ Choosing $z\in Y$ such that 
\begin{equation}\label{equation29}
u(y,s) = sL\left(\frac{d(f(y), \pi^{-1} (z))}{s}\right)  g(z).
\end{equation}
As a consequence, by \eqref{equationPorpL.2}, we deduce that
\begin{equation*}
\begin{aligned}
u(y,t) - u(y,s )  & = \inf_{h\in Y} \left\{ tL\left(\frac{d(f(y), \pi^{-1} (h))}{t} \right)+g(h ) \right\} - sL\left(\frac{d(f(y), \pi^{-1} (z))}{s} \right)-g(z)\\
& \leq  t L\left(\frac{d(f(y), \pi^{-1} (z))}{t} \right)- s L\left(\frac{d(f(y), \pi^{-1} (z))}{s} \right)\\
& \leq 0,
\end{aligned}
\end{equation*} as wished. 
\end{proof}

\begin{rem}
When $L(v)=v^2$ for any $v,$ we have that at every point $x$ where $f$ is differentiable in the classical case, then $u$ is so too. Indeed, thanks to Proposition \ref{proptransformFL} (1), for $y=x+hv$ with $h\in \R$ and $v\in \R^\kappa$ we have 
 \begin{equation*}
\frac{ u(x,t) -u(x+hv,t)}{h} \leq  \frac{2K}{t} \frac{d(f(x),f(x+hv))}{h}.
 \end{equation*} 
\end{rem}

 \subsection{The intrinsic Fenchel-Legendre transform}\label{The intrinsic Fenchel-Legendre transform}  
 It is interesting to notice that in the proof or Proposition \ref{proptransformFL} (3) we have a sort of the Fenchel-Legendre transform in our intrinsic context.

Following \cite{DB09, E10}, we define
\begin{defi}
Let $y\in Y, t>0$ and  $f: Y \to \R^\kappa$ be a continuous section of $\pi.$ The intrinsic Fenchel-Legendre transform of $L$ is a map $L^*: [0,Ils(f)] \to \R$ given by 
\begin{equation}
 L^*(\xi)= L^*_{f, y,t}(\xi):= \sup_{\substack{ z\in Y\, \\ w= \frac{d(f(y), \pi^{-1} (z))}{t} \,  \in \R^+ } } ( \xi w- L\left(w\right)).
\end{equation}
\end{defi}
As in classical way, we define the intrinsic Hamiltonian $H$ associated by $L$ as the  intrinsic Fenchel-Legendre transform of $L$ and so $$H(\xi) :=L^*(\xi), $$  for any $\xi \in [0,Ils(f)].$  Notice that since $\xi \in [0,Ils(f)],$ it makes no sense to consider properties of $H$ when $|\xi| \to \infty,$ as in the usual case. However, inspired by \cite[Theorem 3]{E10}, we give the following properties of this notion.
  \begin{prop}\label{propertiesFLtransformata}
Let $y\in Y, t>0$ and  $f: Y \to \R^\kappa$ be a continuous section of $\pi.$  Then,
\begin{enumerate}
\item If $Y$ is bounded set, then $L^*(\xi) = \xi \frac K t - min_{\substack{ w= \frac{d(f(y), \pi^{-1} (z))}{t} }} L(w).$
\item $H$ is a convex map$;$
\item $L=H^*.$
\end{enumerate}
Thus $H$ is the intrinsic Fenchel-Legendre transform of $L$ and vice versa.
\end{prop}

  \begin{proof}
  
  (1). This follows recall the definition of $K$ (see \eqref{costuniversale}).
  
(2). For each fixed $v,$ the function $p \mapsto p\cdot v -L(v)$ is linear; and consequently, the mapping $$ p\mapsto H(p)=L^*(p) =\sup _{v \in \R^+} pv - L (v),$$ is convex. Indeed, if $\tau \in [0,1]$ and $p,q \in \R,$ it holds
\begin{equation*}
\begin{aligned}
H(\tau p+(1-\tau ) q) &= \sup _v \left( (\tau p+(1-\tau ) q) \right) v - L(v)\\
& \leq \tau \sup _v (pv -L(v)) + (1-\tau ) \sup_v (qv-L(v))\\
& = \tau H(p) + (1-\tau ) H(q),
\end{aligned}
\end{equation*}
as desired.

(3). We prove the inequality $(\geq)$. By $H(p)=L^*(p)=\sup_v (pv-L(v)),$ we get that
\begin{equation*}
H(p)+L(v) \geq pv, \quad \forall p,v \in \R^+
\end{equation*}
and consequently
\begin{equation*}
L(v) \geq \sup_p (pv-H(p)) =H^*(p).
\end{equation*}
On the other hand, the second inequality $(\leq)$ follows noting
\begin{equation}\label{22equationPOI}
\begin{aligned}
H^*(v) &=\sup_p (pv -\sup _r (pr-L(r)))\\
&=\sup_p \inf _r (p(v-r) +L(r)).
\end{aligned}
\end{equation}
Now since $v \mapsto L(v)$ is convex, according with the Section B.1 in  \cite{E10} we know that there is $\ell \in \R$ such that
\begin{equation*}\label{22equationPOI.2}
\begin{aligned}
L(r) \geq L(v) + \ell (r-v), \quad \mbox {with } r \in \R .
\end{aligned}
\end{equation*}
Choosing $p=\ell$ in \eqref{22equationPOI} and using the last inequality, we deduce that
\begin{equation*}
H^*(v)\geq \inf _r (\ell (v-r) +L(r)) \geq L(v),
\end{equation*}
and this complete the proof of the statement.

\end{proof}

 \section{Hamilton-Jacobi type equation} 
 In this section we show that  the 'symmetrized'  Hopf-Lax semigroup is a subsolution of Hamilton-Jacobi type  equation. Here we consider the model case, i.e., when $L(v)=v^2$ and so 
 \begin{equation*}
u(y,t) = \inf_{z\in Y} \left\{ \frac{d^2(f(y), \pi^{-1} (z))}{2t}+g(z) \right\}.
\end{equation*}

 \subsection{ Semicontinuity of $i\mm^\pm$ }  
 
Given a continuous section $f:Y \to \R^\kappa $ of $\pi$, $u(y,t)$ is then defined by the minimum problem \eqref{equationu21}.
We also define:
\begin{equation}
\begin{aligned}
i\mm^+ f(y,t) & := \sup \left\{ \limsup_{n\to \infty}  d(f(y), \pi^{-1} (y_n)) \, :\, (y_n)_n \mbox{ is a minimizing sequence in } \eqref{equationu21} \right\}\\
i\mm^- f(y,t) & := \inf \left\{ \liminf_{n\to \infty}  d(f(y), \pi^{-1} (y_n)) \, :\, (y_n)_n \mbox{ is a minimizing sequence in  }\eqref{equationu21} \right\}\\
\end{aligned}
\end{equation}

The map $(y,t) \mapsto u(y,t)$, $Y\times (0,\infty) \to \R \cup \{ \pm \infty\}$ is obviously upper semicontinuous. The behavior of $u(\cdot, t)$ is not trivial only in the set
\begin{equation*}
\{ y\in Y \,:\,  d(f(y), \pi^{-1} (z)) <\infty \mbox{ for some $z\in Y$ with } g(z)<\infty \},
\end{equation*}
and so we shall restrict our analysis in this set. In particular, it is sufficient to ask that $f$ is a bounded section and $Y$ is a bounded subset of $\R^\kappa$ (see \eqref{costuniversale}). Moreover, in this set, $u(y,t )\in \R \cup \{ -\infty\}$ and so we also define
\begin{equation*}
t_*(y) := \sup \{ t\in \R^+\,:\, u(y,t) >-\infty \},
\end{equation*}
 with the convention $t_*(y)=0$ if $u(y,t) =-\infty $ for all $t>0.$

Notice that by \eqref{costuniversale}, if $Y$ is bounded, then $i\mm^+ f <+\infty .$ Moreover, in general we have that
\begin{equation*}
i\mm^+ f (y,t) \geq i\mm^- f(y,t), \quad \forall y\in Y, t\in \R^+.
\end{equation*} 
Regarding the other inequality we have the following result.
 \begin{prop}\label{theoremHopfLax}
Let $f \in C_b(Y,\R ^\kappa)$ be a section of  $\pi :\R^\kappa \to Y$. Then for every $y\in Y$ and $0<t<s < t_*(y),$ it holds: $i\mm^+ f (y,t) \leq i\mm^- f(y,s).$
\end{prop}

\begin{proof} 
 Fix $0<t<s$ and $y\in Y.$ Let's make this proof under the additional condition that the infimum in $u(y,t)$ and in $u(y,s)$ are both attained and so they are minima (if not one should arrange a bit the proof but it is mainly the same idea). Hence take $y_t, y_s$ minima related respectively to $u(y,t)$ and to $u(y,s)$  and so, we deduce that
  \begin{equation*}
\begin{aligned}
g(y_t) +\frac 1 {2t} d^2(f(y), \pi^{-1} (y_t)) & \leq  g(y_s) +\frac 1 {2t} d^2(f(y), \pi^{-1} (y_s)),\\
g(y_s) +\frac 1 {2s} d^2(f(y), \pi^{-1} (y_s)) & \leq  g(y_t) +\frac 1 {2s} d^2(f(y), \pi^{-1} (y_t)).\\
\end{aligned}
\end{equation*}
Summing up the previous equations, it holds
  \begin{equation*}
\begin{aligned}
\left(\frac 1 {2t} - \frac 1 {2s} \right) d^2(f(y), \pi^{-1} (y_t)) & \leq \left(\frac 1 {2t} - \frac 1 {2s} \right)  d^2(f(y), \pi^{-1} (y_s)),\\
\end{aligned}
\end{equation*}
and so, recall that $s>t$ and $1/s < 1/t,$ we obtain
\begin{equation*}
d^2(f(y), \pi^{-1} (y_t)) \leq d^2(f(y), \pi^{-1} (y_s)).
\end{equation*}
Now let the square root in the last inequality, $d(f(y), \pi^{-1} (y_t)) \leq d(f(y), \pi^{-1} (y_s))$ holds. More precisely, $$d(f(y), \pi^{-1} (y_t)) \leq d(f(y), \pi^{-1} (y_s))$$ is true for every choice $(y_s, y_t)$ into the class of minimizers of $u(y,s)$ and $u(y,t),$ respectively. This gives us the sought inequality and the proof of the statement is complete.

\end{proof}

 \begin{prop}[Semicontinuity of $i\mm^\pm$]\label{propSEMICONT} Let $y_n\to y$ and $t_n \to t \in (0,t_*(y)).$ Then,
 \begin{equation*}
\begin{aligned}
i\mm^- f(y,t) &\leq \liminf_{n\to \infty}  i\mm^- f(y_n,t_n), \\
i\mm^+ f(y,t) &\geq \limsup_{n\to \infty}  i\mm^+ f(y_n,t_n). \\
\end{aligned}
\end{equation*}
In particular, for every $y\in Y$ the map $t\mapsto i\mm^-f(y,t)$ is left continuous in $(0,t_*(y))$ and the map  $t\mapsto i\mm^+f(y,t)$ is right continuous in $(0, t_*(y)).$
\end{prop}

  \begin{proof} 
  We adapt the proof as in \cite[Proposition 3.2]{AGS08.2}. For every $n\in \N,$ let $(y_n^\ell)_{\ell \in \N}$ be a minimizing sequence for $u( y_n,t_n)$ for which the limit of $d(f(y_n), \pi^{-1} (y^\ell_n)) $ as $\ell \to \infty $ equals $i\mm^- (y_n,t_n).$ By Remark \ref{costuniversale}, $\sup _{\ell, n}  d(f(y_n), \pi^{-1} (y_n^\ell)) <+\infty $ and for any $n$ we have that
  \begin{equation*}
\lim_{\ell \to \infty} f(y^\ell_n) +\frac 1 {2t_n}   d^2(f(y_n), \pi^{-1} (y^\ell_n)) = u(y_n, t_n).
\end{equation*}
Moreover, the upper semicontinuity of $(y,t) \mapsto u(y,t)$ gives that $\limsup _n u(x_n, t_n) \leq u(y,t).$ Since $ d(f(y_n), \pi^{-1} (y^\ell_n)) $ is bounded, it follows that $$\sup_\ell | d^2(f(y_n), \pi^{-1} (y^\ell_n))  -  d^2(f(y), \pi^{-1} (y_n^\ell)) | $$ is infinitesimal and so by a diagonal argument we can find a sequence $n\mapsto \ell (n) $ such that
 \begin{equation*}
\begin{aligned}
& \limsup_{ n\to \infty} f(y_n^{\ell(n)})  + \frac 1 {2t}  d^2(f(y), \pi^{-1} (y_n^{\ell(n)})) \leq u(y,t),\\
& | d(f(y_n)), \pi^{-1} (y^{\ell(n)}_n))  -  i\mm^-(y_n, t_n) |  \leq \frac 1 n .\\
\end{aligned}
\end{equation*}
This implies that $y\mapsto y^{\ell (n)}_n $ is a minimizing sequence for $u(y,t),$ therefore 
\begin{equation*}
\mm^- f(y,t) \leq \liminf_{ n\to \infty} d(f(y), \pi^{-1} (y_n^{\ell(n)})) = \liminf_{ n\to \infty} d(f(y_n), \pi^{-1} (y_n^{\ell(n)}))  = \liminf_{ n\to \infty} \mm^-f(y_n, t_n).
\end{equation*}
Notice that in the equality we used that fact that $y_n \to y$ which we will prove in Proposition \ref{theoremHopfLax} (ii).  
In a similar way, if we choose instead sequence $(y^\ell_n)_\ell$ on which the supremum in the definition of $i\mm^+f(y_n, t_n)$ is attained, we obtain the upper semicontinuity property of $i\mm^+f$.

  \end{proof}

  We conclude this section with an easy result when $f$ is an intrinsically Lipschitz section. The proof is similar to \cite[Proposition 3.4]{D22Hopf}.
  
 \begin{prop}\label{theoremHopfLax.2}
Let $f \in ILS(Y, \R ^\kappa)$ and let $L \geq 1$ be its Lipschitz constant. Then, 
$$2t L \geq i\mm^+f(y,t) \geq i\mm^-f(y,t),$$ for every $y\in Y$ and $t\in \R^+.$

\end{prop}

\begin{proof}
We can suppose $u(y,t)< g (y);$ indeed, if not, it must be $$u (y,t) =g(y) \quad  \Rightarrow \quad i\mm^+ f(y,t) =0.$$ Hence we can take a minimizing sequence $(y_n)_n$ for $u(y,t)$ so that definitively 
\begin{equation*}
g(y_n) +\frac 1 {2t} d^2(f(y), \pi^{-1} (y_n)) \leq g(y).
\end{equation*}
Then,  \begin{equation*}
d^2(f(y), \pi^{-1} (y_n))  \leq 2t d(f(y_n), f(y)) \leq 2t L d(f(y), \pi^{-1} (y_n)),
\end{equation*} where in the second inequality we used the Lipschitz property of $f.$ Finally,  dividing for $d(f(y), \pi^{-1} (y_n))$ and taking the limsup in $n$, the thesis follows.
 \end{proof}

 \subsection{The time derivative of $u(\cdot, t)$}
We find a precise estimate of the time derivative of the Hopf-Lax semigroup in terms of $i\mm^\pm f(y,t);$ in order to do this fact, we give an alternative proof of Lipschitz property of $u(\cdot, t ).$ Moreover, we recall that semiconcave map $g$ on an open interval are local quadratic perturbations of concave maps; they inherit from concave functions all pointwise differentiability properties, as existence of right and left derivatives $\frac{d^-}{ dt} g \geq \frac{d^+}{ dt} g$ which is important for the next result.
  \begin{prop}[Time derivative of $u(\cdot, t)$]\label{theoremHopfLaxFormulanew}
The map $(0,t_*(y)) \ni t \mapsto u(\cdot, t)$ is locally Lipschitz and locally semiconcave. For all $t\in (0,t_*(y))$ it satisfies  
\begin{equation}\label{equationIMPO}
\begin{aligned}
\frac{d^-}{ dt} u(y,t) & =- \frac{ (i\mm^-f(y,t) )^2}{ 2t^2},\\
\frac{d^+}{ dt} u(y,t) & =- \frac{ (i\mm^+f(y,t) )^2}{ 2t^2},\\
\end{aligned}
\end{equation} In particular, $t\mapsto u(\cdot, t)$ is differentiable at $t\in (0, t_*(y))$ if and only if $i\mm^+f(y,t)= i\mm^-f(y,t).$
\end{prop}

   \begin{proof} We follows \cite[Proposition 3.3]{AGS08.2} (see also \cite[Proposition 3.9]{D22Hopf}). Let $(y_t^n)_n, (y_s^n)_n$ be minimizing sequences for $u(y,t)$ and $u(y,s)$. We have
   \begin{equation*}\label{equationIMPO.81}
\begin{aligned}
u(y,s) - u(y,t) & \leq  \liminf_{ n \to \infty } \frac{ d^2(f(y)), \pi^{-1} (y_t^n))}{2}\left( \frac 1 s - \frac 1 t \right),\\
 u(y,s) - u(y,t) & \geq  \limsup_{ n \to \infty } \frac{ d^2(f(y)), \pi^{-1} (y_s^n))}{2}\left( \frac 1 s - \frac 1 t \right).
\end{aligned}
\end{equation*} 
Now we have two cases:  $s>t$ or vice versa. In the first case, we get
\begin{equation}\label{equationIMPO.83}
\begin{aligned}
 \frac{ (i\mm^- f(y,s))^2}{2}\left( \frac 1 s - \frac 1 t \right) \leq  u(y,t) - u(y,s) & \leq  \frac{  (i\mm^+ f(y,t))^2 }{2}\left( \frac 1 s - \frac 1 t \right),
\end{aligned}
\end{equation} 
 recalling that $\lim_{s\to t} i\mm^-f(y,s) =i\mm^+f(y,t),$ a division by $s-t$ (noting that $1/s-1/t=-(s-t)/st$) and a limit as $s\to t $ gives the identity for the right derivative in \eqref{equationIMPO}. In a similar way, we can obtain the left derivative.
 
 Moreover, the local Lipschitz continuity follows by \eqref{equationIMPO.83} recalling that $i\mm^\pm f(y,\cdot )$ are locally bounded functions; we easily get the quantitative bound 
 \begin{equation*}
\left\|\frac d {dt} u(y,t) \right\|_{L^\infty (\tau_1, \tau _2)} \leq \frac 1 {2\tau_1^2} \| i\mm^+f(y,\cdot)\|_{L^\infty (\tau_1, \tau _2)},
\end{equation*}
for every $0<\tau_1 < \tau_2 < t_*(y).$ Finally, since the distributional derivative of the map $t\mapsto (i\mm^+f(y,t))^2 /(2t^2)$ is locally bounded from below, we also deduce that $t\mapsto u(y,\cdot)$ is locally semiconcave, as desired. Hence, the proof is complete.
 \end{proof}
 
   \begin{rem}\label{coroll9l}
We want to underline that since $Y\subset \R^\kappa$ is bounded, then the map $(0,t_*(y)) \ni t \mapsto u(y,t)$ is $globally$ Lipschitz. 

  This fact follows from \eqref{equationIMPO.83}, noticing that $iD^\pm f(y,\cdot )$ are globally bounded functions by Remark \ref{costuniversale}.
\end{rem}

 \begin{prop}\label{theoremHJstep1}
Let $f:Y \to \R^\kappa$ be a continuous section of $\pi.$ Then, for $t\in (0, t_*(y)) $ it holds
  \begin{equation}\label{equation3.13ab}
\begin{aligned}
 \limsup_{y\to z } \frac{u(z,t)-u(y,t) }{d(f(y), f (z))}\leq \frac{i\mm^+f(z,t)} {2t}\\
\end{aligned}
\end{equation}
\end{prop}

 \begin{proof}
We use a similar technique as in \cite[Proposition 3.4]{AGS08.2} (see also \cite[Proposition 3.11]{D22Hopf}).

Let $y,z \in Y$ such that $u(y,t) >-\infty.$ We want to show that
\begin{equation}\label{equation3.14}
u(z,t)-u(y,t)  \leq \frac {d(f(z), f (y))}{2t} \left( i\mm^-f(y,t) +  d(f(z), \pi^{-1} (y)) \right).
\end{equation}
Let $(y_n)_n$ be a minimizing sequence for $u(y,t)$ on which the infimum in the definition of $i\mm^-f(y,t)$ is attained, obtaining 
   \begin{equation*}
\begin{aligned}
u(z,t)-u(y,t) & \leq \liminf _{ n\to \infty} \frac{d^2(f(z), \pi^{-1} (y_n))} {2t}  -  \frac{d^2(f(y), \pi^{-1} (y_n))} {2t}\\
& \leq \liminf _{ n\to \infty} \frac{d(f(y), f (z))} {2t} ( d(f(y), \pi^{-1} (y_n))  + d(f(z), \pi^{-1} (y_n))  ) \\
& \leq \frac{d(f(y), f (z))} {2t} ( d(f(z), \pi^{-1} (y)) + i\mm^-f(y,t)),
\end{aligned}
\end{equation*}
where in the last inequality we used the fact that $ \pi^{-1} (y_n) \to  \pi^{-1} (y)$ (see \eqref{equation24giugno}). Hence \eqref{equation3.14} holds. Now dividing both sides of \eqref{equation3.14} by $d(f(y), f (z))$ and taking the $\limsup$ as $y\to z$ we get \eqref{equation3.13ab}, since Proposition \ref{propSEMICONT} yields the upper-semicontinuity of $i\mm^+f.$ 
\end{proof}

  \subsection{$u(\cdot , t)$ as a subsolution of Hamilton-Jacobi type inequality}
In our case, we don't know if $u(\cdot , t)$ is a subsolution of Hamilton-Jacobi inequality as in \cite[Theorem 3.5]{AGS08.2}; however, we have the following corollary.
 \begin{coroll}\label{provaHJno}
Let $f:Y \to \R^\kappa$ be a continuous section of $\pi$ with $Y\subset \R^\kappa$ bounded. Then, for $t\in (0, t_*(y)) $ it holds
  \begin{equation*}
\begin{aligned}
\frac{d^+}{ dt} u(y,t) +  2\limsup_{y\to z }\left( \frac{u(z,t) -u(y,t)}{d(f(y), f (z))} \right)^2 \leq 0.
\end{aligned}
\end{equation*}
\end{coroll}

 \begin{proof}
It is enough to consider Proposition \ref{theoremHopfLaxFormulanew} and \ref{theoremHJstep1}.
 \end{proof}

We conclude this section given a similar result as in \cite{D22Hopf}. Here, we need the Lipschitz property of $f$ in order to get the same statement. This result represents a difference with the case studied in \cite{D22Hopf}; indeed, in  \cite{D22Hopf}, it does not require the additional hypothesis of Lipschitzianity on $f$.
 \begin{coroll}\label{provaHJno.00}
Let $f:Y \to \R^\kappa$ be an intrinsically Lipschitz section of $\pi$ with $Y\subset \R^\kappa$ bounded. Then, for $t\in (0, t_*(y)) $ it holds
  \begin{equation*}
\begin{aligned}
\frac{d^+}{ dt} u(y,t) +  \frac 2 { Ils^2(f) } \limsup_{y\to z }\left( \frac{u(z,t) -u(y,t)}{d(f(y), \pi^{-1}(z))} \right)^2 \leq 0.
\end{aligned}
\end{equation*}
\end{coroll}

 \begin{proof}
The statement follows using \eqref{equation3.14}.
 \end{proof}

 \bibliographystyle{alpha}
\bibliography{DDLD}

\end{document}